\documentclass[twoside,11pt]{article}
\usepackage{graphicx,amsmath,latexsym,amssymb,amsthm}
\usepackage[colorlinks=black,urlcolor=black]{hyperref}
at 9pt  

\date{}
\newtheorem{theorem}{Theorem}[section]

\numberwithin{equation}{section}

\begin{document}
\setlength{\unitlength}{1cm}



\vskip1.0cm

\centerline { \textbf{ Asymptotic properties  of boundary-value
problem
 }}
 \centerline { \textbf{  with transmission conditions }}

\vskip.2cm


\vskip.5cm \centerline {\textbf{ O. Sh. Mukhtarov$^\dag$ and K.
Aydemir$^\dag$  }}

\vskip.5cm

\centerline {$^\dag$Department of Mathematics, Faculty of Science,}
\centerline {Gaziosmanpa\c{s}a University,
 60250 Tokat, Turkey}
\centerline {e-mail : {\tt omukhtarov@yahoo.com,
kadriye.aydemir@gop.edu.tr }}


\vskip.5cm \hskip-.5cm{\small{\bf Abstract :} In this study by
applying an own technique we investigate some asymptotic
approximation properties of new type discontinuous boundary-value
problems, which consists of a Sturm-Liouville equation together with
eigenparameter-dependent boundary and transmission conditions.
\vskip0.3cm\noindent {\bf Keywords :} \ Boundary-value problems,
eigenvalue, eigenfunction, asymptotic formulas, singular point.
{\vskip0.3cm\noindent {\bf AMS subject classifications :  34B24,
34L20 }
\section{\textbf{Introduction}}
Sturm-Liouvilly problems which contained spectral parameter in
boundary conditions form an important part of the spectral theory of
boundary value problems. This type problems has a lot of
applications in  mechanics and physics (see \cite{fu, li, tit} and
references, cited therein). Variety theoretic question of  such type
problems was intensively studied for quite a long time. In the
recent years, there has been increasing interest of this kind
problems which also may have discontinuities in the solution or its
derivative at interior points (see \cite{ba, ch, ka1, ka2, wang1}).
Such problems are connected with discontinuous material properties,
such as heat and mass transfer, vibrating string problems when the
string loaded additionally with points masses, diffraction problems
\cite{ li, tit} and varied assortment of physical transfer problems.
In this paper we shall investigated some asymptotic  approximation
properties of one discontinuous Sturm-Liouville problem for which
the eigenvalue parameter takes part in both differential equation
and boundary conditions and two supplementary transmission
conditions at one interior point are added to boundary conditions.
In particular, we find asymptotic approximation formulas for
eigenvalues and corresponding eigenfunctions. The problems with
transmission conditions  arise in mechanics, such as thermal
conduction problems for a thin laminated plate, which studied in
\cite{tik}. This class of problems essentially differs from the
classical case, and its investigation requires a specific approach
based on the method of separation of variables.  Note that,
eigenfunctions of our problem are discontinuous at the one inner
point of the considered interval, in general.
\section{Statement of the problem and construction\\ of  the
fundamental solutions} Let us consider the boundary value problem,
consisting of the differential equation
\begin{equation}\label{1.1}
T y:=-y^{\prime \prime }(x)+ q(x)y(x)=\lambda y(x)
\end{equation}
on  $[a, c)\cup(c,b]$, with eigenparameter- dependent boundary
conditions
\begin{equation}\label{1.2}
 \tau_{1}(y):=\alpha_{10}y(a)+\alpha_{11}y'(a)=0,
\end{equation}
\begin{equation}\label{1.3}
\tau_{2}(y):=\alpha_{20}y(b)-\alpha_{21}y'(b)+\lambda(\alpha'_{20}y(b)-\alpha'_{21}y'(b))=0
\end{equation}  and the transmission conditions
\begin{equation}\label{1.4}
\tau_{3}(y):=\beta^{-}_{11}y'(c-)+\beta^{-}_{10}y(c-)+\beta^{+}_{11}y'(c+)+\beta^{+}_{10}y(c+)=0,
\end{equation}
\begin{equation}\label{1.5}
\tau_{4}(y):=\beta^{-}_{21}y'(c-)+\beta^{-}_{20}y(c-)+\beta^{+}_{21}y'(c+)+\beta^{+}_{20}y(c+)=0,
\end{equation}
where the potential  $q(x)$ is real-valued function, which
continuous in each of the intervals $[a, c) \  \textrm{and}  (c, b]$
and has a finite limits $q( c\mp0)$, $\lambda$ \ is a complex
spectral parameter, \ $\alpha_{ij}, \ \beta^{\pm}_{ij}, \ (i=1,2 \
 \textrm{and} \ j=0,1), \ \alpha'_{ij} \ (i=2 \ \textrm{and} \ j=0,1)$ are real
numbers. This problem differs from the usual regular Sturm-Liouville
problem in the sense that the eigenvalue parameter $\lambda$ are
contained in both differential equation and boundary conditions and
two supplementary transmission conditions at one interior point are
added to boundary conditions.
Let $$A_{0}= \left[%
\begin{array}{cccc}
  \alpha_{21} & \alpha_{20}  \\
  \alpha'_{21} & \alpha'_{20}
\end{array} %
 \right] \  \textrm{and} \  A= \left[%
\begin{array}{cccc}
  \beta^{-}_{10} & \beta^{-}_{11} & \beta^{+}_{10} & \beta^{+}_{11} \\
  \beta^{-}_{20} & \beta^{-}_{21} & \beta^{+}_{20} & \beta^{+}_{21}
\end{array} %
 \right]. $$
 Denote  the determinant of the matrix $A_{0}$ by $\Delta_{0}$ and the determinant of the k-th
 and
j-th columns of the matrix A  by $\Delta_{kj}$. Note that throughout
this study we shall assume that $ \Delta_{0}>0, \ \Delta_{12}>0 \ \
\textrm{and} \ \Delta_{34}>0. $
 With a view to constructing the characteristic function
we shall define two basic solution $\varphi (x,\lambda )$ and $\psi
(x,\lambda )$ by the following procedure.

At first, let us consider the initial-value problem on the left part
$\left[ a,c\right)$ of the considered interval $[a, c)\cup(c,b]$
\begin{eqnarray}\label{eq}
-y^{\prime \prime }+q(x)y=\lambda y,\text{ \ \ }x\in \left[
a,c\right)\\
\label{tam1} y(a)=\alpha _{11},\text{ \ }y^{\prime
}(a)=-\alpha _{10}
\end{eqnarray}
By virtue of well-known existence and uniqueness theorem of ordinary
differential equation theory  this initial-value problem for each
$\lambda $ has a unique solution $\varphi _{1}(x,\lambda )$. Moreover [\cite{ti}, Teorem 7] this solution is an entire function of $%
\lambda $ for each fixed $x\in \left[ a,c\right).$ By using we shall
investigate the differential equation $(\ref{eq})$ on $(c,b]$
together with special type initial conditions
\begin{eqnarray}\label{tamm1}
y(c) =\frac{1}{\Delta_{12}}(\Delta_{23}\varphi _{1}(c,\lambda
)+\Delta_{24}\varphi^{\prime }_{1}(c,\lambda ))\\
\label{tamm2} y^{\prime }(c)
=\frac{-1}{\Delta_{12}}(\Delta_{13}\varphi _{1}(c,\lambda
)+\Delta_{14}\varphi^{\prime }_{1}(c,\lambda )).
\end{eqnarray}
Define a sequence of functions $y_{n}(x,\lambda ),n=0,1,2,...$ on
interval $\left( c,b\right] $ by the following  equations:
\begin{eqnarray*}
y_{0}(x,\lambda
)&=&\frac{1}{\Delta_{12}}[(\Delta_{23}+c\Delta_{13})\varphi
_{1}(c,\lambda )+(\Delta_{24}+c\Delta_{14})\varphi^{\prime
}_{1}(c,\lambda )\nonumber\\&+&(-\Delta_{13}\varphi _{1}(c,\lambda
)+\Delta_{14}\varphi^{\prime }_{1}(c,\lambda ))x]\nonumber\\
y_{n}(x,\lambda )&=&y_{0}(x,\lambda
)+\int\limits_{c}^{x}(x-z)(q(z)-\lambda )y_{n-1}(z,\lambda
)dz,\text{ \ }n=1,2,...
\end{eqnarray*}
It is easy to see that each of  $y_{n}(x,\lambda )$ is an entire
function of $\lambda $ for each $\left( c,b\right] $ Consider the
series
\begin{equation}
y_{0}(x,\lambda )+\sum\limits_{n=1}^{\infty }(y_{n}(x,\lambda
)-y_{n-1}(x,\lambda ))  \label{(3.11)}
\end{equation}
Denoting $ q_{1}=\max_{x\in (c,b]}|q(x)|\text{ \ and \ }Y(\lambda
)=\max_{x\in (c,b]}|y_{0}(x,\lambda )|, $ we get\\
 $
|y_{n}(x,\lambda )-y_{n-1}(x,\lambda )|\leq \frac{1}{(2n)!}Y(\lambda
)(q_{1}+|\lambda |^{n})(x-c)^{2n}$ for each $n=1,2,...$.\\ Because
of this inequality
 the series
$(\ref{(3.11)})$ is uniformly convergent \ with respect to the
variable $x$ on $\left( c,b\right] $, and with respect to the
variable $\lambda $ on every closed bar $|\lambda |\leq R.$ Let
$\varphi _{2}(x,\lambda )$ be the sum of the series
$(\ref{(3.11)}).$ Consequently $\varphi _{2}(x,\lambda )$ is an
entire function of $\lambda $ for each fixed $x\in ( c,b].$ Since
for $n\geq 2$
$$
y_{n}^{\prime }(x,\lambda )-y_{n-1}^{\prime }(x,\lambda
)=\int\limits_{0}^{x}(q(z)-\lambda )(y_{n-1}(z,\lambda
)-y_{n-2}(z,\lambda ))dz
$$
and
$$
y_{n}^{\prime \prime }(x,\lambda )-y_{n-1}^{\prime \prime
}(x,\lambda )=(q(x)-\lambda )(y_{n-1}(x,\lambda )-y_{n-2}(x,\lambda
))
$$
the first and second differentiated series also converge uniformly
with respect to $x.$ Taking into account the last equality we have
\begin{eqnarray*}
\varphi _{2}^{\prime \prime }(x,\lambda ) &=&y_{1}^{\prime \prime
}(x,\lambda )+\sum\limits_{n=2}^{\infty }(y_{n}^{\prime \prime
}(x,\lambda
)-y_{n-1}^{\prime \prime }(x,\lambda )) \\
&=&(q(x)-\lambda )y_{1}(x,\lambda ) \\
&&+\sum\limits_{n=2}^{\infty }(q(x)-\lambda
)(y_{n}(x,\lambda )-y_{n-1}(x,\lambda )) \\
&=&(q(x)-\lambda )\varphi _{2}(x,\lambda ),
\end{eqnarray*}
so $\varphi _{2}(x,\lambda )$ satisfies the equation $(\ref{eq}).$
Moreover, since each $y_{n}(x,\lambda )$ satisfies the initial
conditions $(\ref{tamm1} )$ and $(\ref{tamm2})$, then the function
$\varphi _{2}(x,\lambda )$ satisfies the initial conditions
$(\ref{tamm1} )$ and $(\ref{tamm2})$. Consequently, the function
$\varphi (x,\lambda )$ defined by
\begin{equation}
\varphi (x,\lambda )=\{
\begin{array}{c}
\varphi _{1}(x,\lambda )\text{ \ for }x\in \lbrack a,c) \\
\varphi _{2}(x,\lambda )\text{ \ for }x\in (c,b].%
\end{array}
\label{(3.16)}
\end{equation}
 satisfies equation $(\ref{1.1})$,  the first boundary condition
$(\ref{1.2})$ and the both transmission conditions $(\ref{1.4})$ and $(%
\ref{1.5})$. By applying the same technique we can prove that  for
any $\lambda \in C$ the differential equation $(\ref{1.1})$ has such
 solution\begin{equation}
\psi (x,\lambda )=\{
\begin{array}{c}
\psi _{1}(x,\lambda )\text{ \ for }x\in \lbrack a,c) \\
\psi _{2}(x,\lambda )\text{ \ for }x\in (c,b].%
\end{array}
\end{equation}
which satisfies the initial condition $(\ref%
{1.3}),$ the both transmission conditions $(\ref{1.4})-(\ref{1.5})$
for each $x\in \lbrack a,c)\cup(c,b]$ and  is an entire function of
$\lambda $ for each fixed $x\in \lbrack a,c)\cup(c,b]$. Below, for
shorting  we shall use also  notations; $\varphi _{i}(x,\lambda
):=\varphi _{i\lambda}, \ \psi _{i}(x,\lambda ):=\psi _{i\lambda}.$
\section{Some asymptotic approximation formulas
for fundamental solutions} Let $\lambda =s^{2}$. By applying the
method of variation of parameters we can prove that the next
integral and
integro-differential equations are hold for $k=0$ and $k=1.$%
\begin{eqnarray}\label{(4.2)}
\frac{d^{k}}{dx^{k}}\varphi _{1\lambda}(x ) &=&\alpha _{11}\frac{d^{k}}{dx^{k}}%
\cos \left[ s\left( x-a\right) \right]
-\frac{a_{10}}{s}\frac{d^{k}}{dx^{k}}\sin \left[ s\left( x-a\right)
\right]
\notag \\
&&+\frac{1}{s}\int\limits_{a}^{x}\frac{d^{k}}{dx^{k}}\sin \left[ %
s\left( x-z\right) \right] q(z)\varphi _{1}(z,\lambda )dz \\
\label{(4.22)} \frac{d^{k}}{dx^{k}}\psi _{1\lambda}(x )
&=&-\frac{1}{\Delta_{34}}(\Delta_{14}\psi _{2}(c,\lambda
)+\Delta_{24} \psi^{\prime }_{2}(c,\lambda ))
\frac{d^{k}}{dx^{k}}\cos \left[
s(x-c)\right] \nonumber \\
&&+\frac{1}{s \Delta_{34}}(\Delta_{13}\psi _{2}(c,\lambda
)+\Delta_{23}\psi^{\prime }_{2}(c,\lambda ))\frac{d^{k}}{dx^{k}}\sin
\left[ s(x-c)
\right]  \notag \\
&&+\frac{1}{s}\int\limits_{x}^{c}\frac{d^{k}}{dx^{k}}\sin \left[
s\left( x-z\right) \right] q(z)\psi _{1}(z,\lambda )dz
\label{(4.22)}
\end{eqnarray}
for $x \in [a,c)$ and
\begin{eqnarray}\label{(4.a)}
\frac{d^{k}}{dx^{k}}\varphi_{2\lambda}(x )
&=&\frac{1}{\Delta_{12}}(\Delta_{23}\varphi _{1}(c,\lambda
)+\Delta_{24}\varphi^{\prime }_{1}(c,\lambda ))
\frac{d^{k}}{dx^{k}}\cos \left[
s(x-c)\right]  \notag \\
&&-\frac{1}{s \Delta_{12}}(\Delta_{13}\varphi _{1}(c,\lambda
)+\Delta_{14}\varphi^{\prime }_{1}(c,\lambda
))\frac{d^{k}}{dx^{k}}\sin \left[ s(x-c)
\right]  \notag \\
&&+\frac{1}{s}\int\limits_{c}^{x}\frac{d^{k}}{dx^{k}}\sin \left[
s\left( x-z\right) \right] q(z)\varphi _{2}(z,\lambda )dz
\\ \label{(4.21)}
\frac{d^{k}}{dx^{k}}\psi _{2\lambda}(x )&=&(\alpha _{21}+\lambda \alpha' _{21})\frac{d^{k}}{dx^{k}}%
\cos \left[ s\left( x-b\right) \right] +\frac{1}{s}(\alpha
_{20}+\lambda \alpha' _{20})\frac{d^{k}}{dx^{k}}\sin \left[ s\left(
x-b\right) \right]
\notag \\
&&+\frac{1}{s}\int\limits_{x}^{b}\frac{d^{k}}{dx^{k}}\sin \left[ %
s\left( x-z\right) \right] q(z)\psi _{1}(z,\lambda )dz
\end{eqnarray}
for $x \in (c,b]$. Now we are ready to prove the following theorems.
\begin{theorem} \label{(4.n)}
Let $\lambda =s^{2}$, $Ims=t.$ Then \ if $\alpha _{11}\neq 0$
\begin{eqnarray}
\frac{d^{k}}{dx^{k}}\varphi _{1\lambda}(x ) &=&\alpha _{11}\frac{d^{k}}{dx^{k}}%
\cos \left[ s\left( x-a\right) \right] +O\left( \left| s\right|
^{k-1}e^{_{\left| t\right| (x-a)}}\right)
\label{(4.3)} \\
\frac{d^{k}}{dx^{k}}\varphi _{2\lambda}(x ) &=&-\frac{\Delta_{24}}{%
\Delta_{12}}\alpha _{11}s\sin \left[ s\left( c-a\right) \right]
\frac{d^{k}}{dx^{k}}\cos \left[ s\left( x-c\right) \right] \notag
\\
&&+O\left(|s| ^{k} e^{\left| t\right|(x-a)}\right) \label{(4.4)}
\end{eqnarray}
as $\left| \lambda \right| \rightarrow \infty $, while if $\alpha _{11}=0$%
\begin{eqnarray}
\frac{d^{k}}{dx^{k}}\varphi _{1\lambda}(x ) &=&-\frac{\alpha_{10}}{s}%
\frac{d^{k}}{dx^{k}}\sin \left[ s(x-a)\right] +O\left( \left|
s\right| ^{k-2}e^{\left| t\right| (x-a)}\right) \label{(4.5)}
\\
\frac{d^{k}}{dx^{k}}\varphi _{2\lambda}(x ) &=&-\frac{\Delta_{24}}{%
\Delta_{12}}\alpha _{10}\cos  \left[ s\left( c-a\right) \right] \frac{d^{k}}{dx^{k}} \cos \left[ s(x-c)\right]  \notag \\
&&+O\left( \left| s\right| ^{k-1}e^{\left| t\right| (x-a)}{%
}\right)  \label{(4.6)}
\end{eqnarray}
as $\left| \lambda \right| \rightarrow \infty $ ($k=0,1)$. Each of
this asymptotic equalities hold uniformly for $x.$
\end{theorem}
\begin{proof}
The asymptotic formulas for $(\ref{(4.3)})$ in $(\ref{(4.4)})$
follows immediately from the Titchmarsh's Lemma on the asymptotic
behavior of $\varphi
_{\lambda }(x)$ (\cite{ti}, Lemma 1.7). But the corresponding formulas for $%
\varphi _{2\lambda}(x )$ need individual consideration. Let $\alpha
_{11}\neq 0$. Substituting $(\ref{(4.3)})$ in $(\ref{(4.a)})$ we
have the next ''asymptotic integral equation''
\begin{eqnarray}\label{(4.t)}
 \varphi _{2\lambda}(x ) &=& \frac{1}{\Delta_{12}}\alpha _{11}
\big[\Delta_{23}\cos s(c-a)\cos
s(x-c)-\Delta_{24}s \sin s(c-a)%
\cos s(x-c) \notag \\&&-\frac{\Delta_{13}}{s}\cos s (c-a) \sin
s(x-c)+\frac{\Delta_{14}}{s}\sin s(c-a) \sin s(x-c)
\big] \notag \\
&&+\frac{1}{s}\int\limits_{c}^{x}\sin \left[ s(x-z)\right] q(z)\phi
_{2}(z,\lambda )dz +O(e^{\left| t\right| (x-a)})
\end{eqnarray}
Multiplying by $e^{-\left| t\right| (x-a) }$ and denoting \ \ $
Y(x,\lambda )=e^{-\left| t\right| (x-a)} \varphi _{2\lambda}(x ) $
we get
\begin{eqnarray}
Y(x,\lambda ) &=&\frac{1}{\Delta_{12}}\alpha _{11}e^{-\left|
t\right| (x-a)} \big[\Delta_{23}\cos s(c-a)\cos
s(x-c)-\Delta_{24}s \sin s(c-a)%
\cos s(x-c) \notag \\&&-\frac{\Delta_{13}}{s}\cos s (c-a) \sin
s(x-c)+\frac{\Delta_{14}}{s}\sin s(c-a) \sin s(x-c)
\big] \notag \\
&&+\frac{1}{s}\int\limits_{c}^{x}\sin \left[ s(x-z)\right] q(z)
e^{-\left| t\right| (x-z)} Y(z,\lambda )dz +O(1)  \label{(4.7)}
\end{eqnarray}
Denoting \ \ $
Y(\lambda )=\underset{x\in \lbrack c,b]}{\max }|Y(x,\lambda )|\text{ and }%
\widetilde{q}=\int\limits_{c}^{b}|q(z)|dz $ from the last equation
we have $Y(\lambda )=O(1)  \ \textrm{as} \left| \lambda \right|
\rightarrow \infty ,$ so $ \varphi _{2\lambda}(x )=O(e^{\left|
t\right| (x-a)}) $. Substituting in $(\ref{(4.t)})$  we obtain
$(\ref{(4.4)})$ for the case $k=0$. The case $k=1$ of the
$(\ref{(4.4)})$ follows at once on differentiating $(\ref{(4.a)})$
and making the same procedure as in the case $k=0$. The proof of
$(\ref{(4.5)})$ in $(\ref{(4.6)})$ is similar.
\end{proof}
Similarly  we can easily obtain the following Theorem for
$\psi_{i}(x,\lambda) (i=1,2).$
\begin{theorem} \label{(c1)}
Let $\lambda =s^{2}$, $Ims=t.$ Then \ if $\alpha' _{21}\neq 0$
\begin{eqnarray}
\frac{d^{k}}{dx^{k}}\psi _{2\lambda}(x ) &=&\alpha' _{21}s ^{2}\frac{d^{k}}{dx^{k}}%
\cos \left[ s\left( b-x\right) \right] +O\left( \left| s\right|
^{k+1}e^{_{\left| t\right| (b-x)}}\right)
\label{(c2)} \\
\frac{d^{k}}{dx^{k}}\psi _{1\lambda}(x ) &=&-\frac{\Delta_{24}}{%
\Delta_{34}}\alpha' _{21}s ^{3}\sin  \left[ s\left( b-c\right)
\right] \frac{d^{k}}{dx^{k}}\cos \left[ s\left( x-c\right) \right]
\notag
\\
&&+O\left(|s| ^{k+2} e^{\left| t\right|(b-x)}\right) \label{(4.n1)}
\end{eqnarray}
as $\left| \lambda \right| \rightarrow \infty $, while if $\alpha'_{21}=0$%
\begin{eqnarray}
\frac{d^{k}}{dx^{k}}\psi _{2\lambda}(x ) &=&-a'_{20}s%
\frac{d^{k}}{dx^{k}}\sin \left[ s(b-x)\right] +O\left( \left|
s\right| ^{k}e^{\left| t\right| (b-x)}\right) \label{(c3)}
\\
\frac{d^{k}}{dx^{k}}\psi _{1\lambda}(x) &=&-\frac{\Delta_{24}}{%
\Delta_{34}}\alpha' _{20}s^{2}\cos  \left[ s\left( b-c\right) \right] \frac{d^{k}}{dx^{k}} \cos \left[ s(x-c)\right]  \notag \\
&&+O\left( \left| s\right| ^{k+1}e^{\left| t\right| (b-x)}{%
}\right)  \label{(c4)}
\end{eqnarray}
as $\left| \lambda \right| \rightarrow \infty $ ($k=0,1)$. Each of
this asymptotic equalities hold uniformly for $x.$
\end{theorem}
\section{Asymptotic behaviour
 of eigenvalues and\\ corresponding eigenfunctions%
} It is well-known from ordinary differential equation theory that
the Wronskians $W[\varphi _{1\lambda},\psi _{1\lambda}]_{x}$ and
$W[\varphi _{2\lambda}, \psi _{2\lambda}]_{x}$ are independent of
variable $x.$ Denoting $w_{i}(\lambda )=W[\varphi _{i\lambda},\psi
_{i\lambda}]_{x}$ we have
\begin{eqnarray*}
w_{2}(\lambda ) &=&\varphi_{2}(c,\lambda )\psi _{2}^{\prime
}(c,\lambda )-\varphi
_{2}^{\prime }(,\lambda )\psi _{2}(c,\lambda ) \\
 &=&\frac{\Delta_{34}}{\Delta_{12}}(\varphi _{1}(c,\lambda )\psi
_{1}^{\prime }(c,\lambda )-\varphi
_{1}^{\prime }(c,\lambda )\psi _{1}(c,\lambda )) \\
&=&\frac{\Delta_{34}}{\Delta_{12}} w_{1}(\lambda )
\end{eqnarray*}
Denote $ \omega(\lambda):= \Delta_{34} \omega_{1}(\lambda) =
\Delta_{12} \ \omega_{2}(\lambda). $

By the same technique as in \cite{ka1} we can prove the following
theorem
\begin{theorem}
The eigenvalues of the problem $(\ref{1.1})$-$(\ref{1.5})$ are
consist of the zeros of the function $w(\lambda ).$\label{t2}
\end{theorem}
Now by modifying the standard method we  prove that all eigenvalues
of the problem $(\ref{1.1})-(\ref{1.5})$ are real.
\begin{theorem}
All eigenvalues of the problem $(\ref{1.1})-(\ref{1.5})$ are real.
\end{theorem}
\begin{proof}
\end{proof}
Since the Wronskians of $\varphi _{2\lambda}(x )$ and $\psi
_{2\lambda}(x )$ are independent of $x$, in particular, by putting
$x=a$ we have
\begin{eqnarray}\label{(ko)}
w(\lambda ) &=&\varphi _{1}(a,\lambda )\psi_{1}^{\prime }(a,\lambda
)-\varphi
_{1}^{\prime }(a,\lambda )\psi _{1}(a,\lambda ) \notag \\
&=&\alpha _{11}\psi _{1}^{\prime }(a,\lambda )+\alpha _{10}\psi
_{1}(a,\lambda ).
\end{eqnarray}
 Let $\lambda =s^{2}$, $Ims=t.$ By substituting $(\ref{(c2)})$ and
$(\ref{(c4)})$ in $(\ref{(ko)})$ we obtain easily the following
asymptotic representations\\ \textbf{(i)} If $\alpha _{21} ^{\prime
}\neq 0$ and $\alpha _{11}\neq 0$, then
\begin{equation}
w(\lambda )=\Delta_{24}\alpha _{11}\alpha _{21}^{\prime }%
s^{4}\sin \left[ s\left( b-c\right) \right] \sin\left[ s\left(
a-c\right) \right] +O\left( \left| s\right| ^{3}e^{\left| t\right|
(b-a) }\right)  \label{(4.15)}
\end{equation}
\textbf{(ii)} If $\alpha _{21} ^{\prime }\neq  0$ and $\alpha _{11}=
0$, then
\begin{equation}
w(\lambda )=-\Delta_{24}\alpha _{10}\alpha _{21}^{\prime }%
s^{3}\sin \left[ s\left( b-c\right) \right] \cos\left[ s\left(
a-c\right) \right] +O\left( \left| s\right| ^{2}e^{\left| t\right|
(b-a) }\right)  \label{(4.16)}
\end{equation}
\textbf{(iii)} If $\alpha _{21} ^{\prime }= 0$ and $\alpha _{11}\neq
0$, then
\begin{equation}
w(\lambda )=\Delta_{24}\alpha _{11}\alpha _{20}^{\prime }%
s^{3}\cos \left[ s\left( b-c\right) \right] \sin \left[ s\left(
a-c\right) \right] +O\left( \left| s\right| ^{2}e^{\left| t\right|
(b-a) }\right)  \label{(4.17)}
\end{equation}
\textbf{(iv)} If $\alpha _{21} ^{\prime }=0$ and $\alpha _{11}= 0$,
then
\begin{equation}
w(\lambda )=-\Delta_{24}\alpha _{10}\alpha _{20}^{\prime }%
s^{2}\cos \left[ s\left( b-c\right) \right] \cos\left[ s\left(
a-c\right) \right] +O\left( \left| s\right| e^{\left| t\right| (b-a)
}\right)  \label{(4.18)}
\end{equation}
Now we are ready to derived the needed asymptotic formulas for
eigenvalues and  eigenfunctions.
\begin{theorem}
The boundary-value-transmission problem $(\ref{1.1})$-$(\ref{1.5})$
has an precisely numerable many real eigenvalues, whose behavior may
be expressed by two sequence $\left\{ \lambda _{n,1}\right\} $ and
$\left\{
\lambda _{n,2}\right\} $ with following asymptotic as $n\rightarrow \infty $%
\textbf{(i)} If $\alpha _{21} ^{\prime }\neq 0$ and $\alpha
_{11}\neq 0$, then
\begin{equation}
s_{n,1}=(n-2)\frac{\pi}{b-c} +O\left( \frac{1}{n}\right) ,\text{ }%
s_{n,2}=\frac{n \pi}{a-c} +O\left( \frac{1}{n}\right) ,
\label{(5.1)}
\end{equation}
\textbf{(ii)} If $\alpha _{21} ^{\prime }\neq  0$ and $\alpha _{11}=
0$, then
\begin{equation}
s_{n,1}=(n+\frac{1}{2})\frac{\pi}{b-c} +O\left( \frac{1}{n}\right) ,\text{ }%
s_{n,2}=\frac{\pi}{a-c} (n-1)+O\left( \frac{1}{n}\right) ,
\label{(5.2)}
\end{equation}
\textbf{(iii)} If $\alpha _{21} ^{\prime }= 0$ and $\alpha _{11}\neq
0$, then
\begin{equation}
s_{n,1}=(n-1)\frac{\pi}{b-c} +O\left( \frac{1}{n}\right) ,\text{
}s_{n,2}=\frac{\pi}{a-c} (n+\frac{1}{2})+O\left( \frac{1}{n}\right)
, \label{(5.3)}
\end{equation}
\textbf{(iv)} If $\alpha _{21} ^{\prime }= 0$ and $\alpha _{11}= 0$,
then
\begin{equation}
s_{n,1}=(n-\frac{1}{2})\frac{\pi}{b-c} +O\left( \frac{1}{n}\right) ,\text{ }%
s_{n,2}=\frac{\pi }{a-c}(n+\frac{1}{2})+O\left( \frac{1}{n}\right)
\label{(5.4)}
\end{equation}
where $\lambda _{n,1}=s_{n,1}^{2}$ \label{t4}$, \lambda
_{n,2}=s_{n,2}^{2}$
\end{theorem}
\begin{proof}
\end{proof}
Using this asymptotic expression of eigenvalues we can easily obtain
the corresponding asymptotic expressions for eigenfunctions of the
problem $(\ref{1.1})$-$(\ref{1.5})$. Denote  the corresponding
eigenfunction of the problem by
$$
\widetilde{\varphi}_{n,i}=\{
\begin{array}{c}
\varphi _{1\lambda _{n,i}}(x)\text{ \ for }x\in \lbrack a,c) \\
\varphi _{2\lambda _{n,i}}(x,)\text{ \ for }x\in (c,b].%
\end{array}
$$
 Recalling that $\varphi_{\lambda _{n,i}}(x)$ is an
eigenfunction
 according to the eigenvalue $\lambda _{n},$ by putting (\ref{(5.1)}) in the (\ref{(4.3)}) for $k=0$ we get
\begin{equation*}
\widetilde{\varphi} _{n,1}(x)=\left\{
\begin{array}{ll}
\begin{array}{l}
\alpha_{11}\cos \left[
  (n-2)\pi\frac{
 (x-a)}{(b-c)}\right]+O\left(\frac{1}{n}%
\right)
\\
\end{array}
\begin{array}{l}
\textrm{ for }x\in \lbrack a,c)
\end{array}
\\
-\alpha _{11}\Delta_{24} \frac{ \pi (n-2)}{\Delta_{12}(b-c)} \sin
\left[
  (n-1)\pi\frac{
 (c-a)}{(b-c)}\right]\cos \left[
  (n-2)\pi\frac{
 (x-c)}{(b-c)}\right]
\notag
\\
+O\left(1%
\right) \textrm{ for }x\in (c,b]
\end{array}%
\right.
\end{equation*}%
if $\alpha _{21} ^{\prime }\neq 0$ and $\alpha _{11}\neq 0$.
Similarly,  by putting (\ref{(5.1)}) in the (\ref{(4.4)}) for $k=0$
yields
\begin{equation*}
\widetilde{\varphi} _{n,2}(x)=\left\{
\begin{array}{l}
\alpha_{11}\cos \left[
  n\pi\frac{
 (x-a)}{(a-c)}\right]+O\left(\frac{1}{n}%
\right) \ \textrm{ for }x\in \lbrack a,c) \\
-\alpha _{11}\Delta_{24} \frac{ n \pi }{\Delta_{12}(a-c)} \sin
\left[
  n\pi\right]\cos \left[
 n\pi\frac{
 (x-c)}{(a-c)}\right]
\notag \\
+O\left( 1\right) \textrm{ for }x\in (c,b] \\
\end{array}%
\right.
\end{equation*}
Similar expressions  are as follows:\\
  If $\alpha _{21} ^{\prime }\neq 0$ and $\alpha _{11}=
0$, then
\begin{equation*}
\widetilde{\varphi} _{n,1}(x)=\left\{
\begin{array}{ll}
\begin{array}{l}
-\alpha_{10}\frac{
 (b-c)}{\pi(n+\frac{1}{2})}\sin \left[\pi(n+\frac{1}{2})\frac{(x-a)}{(b-c)}\right]+O\left(\frac{1}{n^{2}}
\right)
\\
\end{array}
\begin{array}{l}
\text{ for }x\in \lbrack a,c)
\end{array}
\\
-\alpha _{10}\frac{\Delta_{24}}{\Delta_{12}}  \cos
\left[\pi(n+\frac{1}{2})\frac{(c-a)}{(b-c)}\right]\cos \left[\pi
(n+\frac{1}{2}) \frac{(x-c)}{(b-c)}\right]
\notag \\+O\left(\frac{1}{n}%
\right) \ \textrm{ for }x\in (c,b]
\end{array}%
\right.
\end{equation*}%
and
\begin{equation*}
\widetilde{\varphi} _{n,2}(x)=\left\{
\begin{array}{l}
-\alpha_{10}\frac{(a-c)}{\pi(n-1)}\sin \left[
  (n-1)\pi\frac{
 (x-a)}{(a-c)}\right]+O\left(\frac{1}{n^{2}}%
\right) \ \textrm{ for }x\in \lbrack a,c) \\
-\alpha _{10} \frac{ \Delta_{24}}{\Delta_{12}} \cos \left[
  \pi(n-1)\right]\cos \left[
  \pi(n-1)\frac{
 (x-c)}{(a-c)}\right]
\notag \\
+O\left( \frac{1}{n}\right) \ \textrm{ for }x\in (c,b] \\
\end{array}%
\right.
\end{equation*}
}  If $\alpha _{21} ^{\prime }= 0$ and $\alpha _{11}\neq 0$, then \
\begin{equation*}
\widetilde{\varphi} _{n,1}(x)=\left\{
\begin{array}{ll}
\begin{array}{l}
\alpha_{11}\cos \left[
  \pi(n-1)\frac{
 (x-a)}{(b-c)}\right]+O\left(\frac{1}{n}
\right)
\end{array}
\begin{array}{l}
\textrm{ for }x\in \lbrack a,c)
\end{array}
\\
-\alpha _{11}\Delta_{24} \frac{ \pi (n-1)}{\Delta_{12}(b-c)} \sin
\left[
  (n-1)\pi\frac{
 (c-a)}{(b-c)}\right]\cos \left[
  (n-1)\pi\frac{
 (x-c)}{(b-c)}\right]
\notag
\\
+O\left(1%
\right) \textrm{ for }x\in (c,b]%
\end{array}%
\right.
\end{equation*}
and
\begin{equation*}
\widetilde{\varphi} _{n,2}(x)=\left\{
\begin{array}{l}
\alpha_{11}\cos \left[
  \pi(n+\frac{1}{2})\frac{
 (x-a)}{(a-c)}\right]+O\left(\frac{1}{n}%
\right) \textrm{ \ \ for }x\in \lbrack a,c)
\\
{\normalsize {}}-\alpha _{11}\pi \frac{
\Delta_{24}(n+\frac{1}{2})}{\Delta_{12}(a-c)} \sin \left[
  \pi(n+\frac{1}{2})\right]\cos \left[
  (n+\frac{1}{2})\pi\frac{
 (x-c)}{(a-c)}\right]
\notag \\
+O\left( 1\right) \textrm{ for }x\in (c,b] \\
\end{array}%
\right.
\end{equation*}

 If $\alpha _{21} ^{\prime }= 0$ and $\alpha _{11}= 0$,
then
\begin{equation*}
\widetilde{\varphi} _{n,1}(x)=\left\{
\begin{array}{ll}
\begin{array}{l}
-\alpha_{10}\frac{(b-c)}{(n-\frac{1}{2})\pi}\sin
\left[(n-\frac{1}{2})\pi
  \frac{
 (x-a)}{(b-c)}\right]+O\left(\frac{1}{n^{2}}%
\right)
\end{array}
\begin{array}{l}
\textrm{ for }x\in \lbrack a,c)
\end{array}
\\
-\alpha _{10} \frac{ \Delta_{24}}{\Delta_{12}} \cos \left[
  (n-\frac{1}{2})\pi\frac{
 (c-a)}{(b-c)}\right]\cos \left[
  (n-\frac{1}{2})\pi\frac{
 (x-c)}{(b-c)}\right]
\notag
\\
+O\left(\frac{1}{n}%
\right) \ \ \textrm{ for }x\in (c,b]%
\end{array}%
\right.
\end{equation*}%
and
\begin{equation*}
\widetilde{\varphi} _{n,2}(x)=\left\{
\begin{array}{ll}
\begin{array}{l}
-\alpha_{10}\frac{(a-c)}{(n+\frac{1}{2})\pi}\sin
\left[(n+\frac{1}{2})\pi
  \frac{
 (x-a)}{(a-c)}\right]+O\left(\frac{1}{n^{2}}%
\right) ,
\end{array}

\begin{array}{l}
\textrm{ for }x\in \lbrack a,c)
\end{array}
\\
-\alpha _{10} \frac{ \Delta_{24}}{\Delta_{12}} \cos \left[
  (n+\frac{1}{2})\pi\right]\cos \left[
  (n+\frac{1}{2})\pi\frac{
 (x-c)}{(a-c)}\right]
\notag
\\
+O\left(\frac{1}{n}%
\right)  \textrm{ for }x\in (c,b]%
\end{array}%
\right.
\end{equation*}%
All this asymptotic approximations are hold uniformly for $x.$%

\end{document}